\documentclass[12pt,a4paper,reqno]{amsart}

\usepackage[active]{srcltx}
\usepackage{a4wide}
\usepackage{amssymb, amsmath, mathtools}
\usepackage{graphicx}
\usepackage{float}
\usepackage[active]{srcltx}
\usepackage[ansinew]{inputenc}

\usepackage{hyperref}

\theoremstyle{plain}
\newtheorem{theorem}{Theorem}[section]
\newtheorem{maintheorem}{Theorem}
\newtheorem{lemma}[theorem]{Lemma}
\newtheorem{proposition}[theorem]{Proposition}

\theoremstyle{remark}

\newtheorem{definition}{Definition}

\newtheorem{remark}[theorem]{Remark}

\def\R{\ensuremath{\mathbb R}}
\def\N{\ensuremath{\mathbb N}}

\def\top{\operatorname{top}}
\def\dist{\operatorname{dist}}

\numberwithin{equation}{section}

\begin{document}
\title[A variational principle for impulsive semiflows]{A variational principle for impulsive semiflows}

\author[J. F. Alves]{Jos\'{e} F. Alves}
\address{Jos\'{e} F. Alves\\ Centro de Matem\'{a}tica da Universidade do Porto\\ Rua do Campo Alegre 687\\ 4169-007 Porto\\ Portugal}
\email{jfalves@fc.up.pt} \urladdr{http://www.fc.up.pt/cmup/jfalves}

\author[M. Carvalho]{Maria Carvalho}
\address{Maria Carvalho\\ Centro de Matem\'{a}tica da Universidade do Porto\\ Rua do
Campo Alegre 687\\ 4169-007 Porto\\ Portugal}
\email{mpcarval@fc.up.pt}

\author[C.H. V\'asquez]{Carlos H. V\'asquez}
\address{Carlos Vasquez\\ Instituto de Matem\'atica\\ Pontif\'icia Universidad Cat\'olica de Valpara\'iso\\
Valpara\'iso, Chile}
\email{carlos.vasquez@ucv.cl}

\date{\today}
\thanks{The authors were partially funded by the
European Regional Development Fund through the program COMPETE and by the Portuguese Government through the FCT - Funda{\c c}{\~a}o para a Ci{\^e}ncia e a Tecnologia under the project PEst-C/MAT/UI0144/2013.
JFA was also partially supported by Funda\c c\~ao Calouste Gulbenkian and the project PTDC/MAT/120346/2010.
CHV  was supported by the Center of Dynamical Systems and Related
Fields, c\'odigo ACT1103 PIA - Conicyt and Proyecto Fondecyt 1130547.}

\begin{abstract}
We consider impulsive semiflows defined on compact metric spaces and deduce a variational principle. In particular, we generalize the classical notion of topological entropy to our setting of discontinuous  semiflows.
\end{abstract}

\maketitle

\setcounter{tocdepth}{2}

\tableofcontents 

\section{Introduction}

An impulsive semiflow  is built from a continuous semiflow $\varphi$ acting on a compact metric space $X$, which governs the state of the system between impulses, and a compact set $D \subset X$ where the semiflow $\varphi$ experiences some drift specified by a continuous function $I:D\to X$. Dynamical systems with impulse perturbations seem to be an efficient mathematical tool to describe real world phenomena that exhibit abrupt transitions in their phase space; see the introduction of \cite{AC14} and references therein for a thorough list of applications of impulsive semiflows, particularly in Life Sciences and Physics.

A major problem in this field is inherent to the dynamics: an impulsive semiflow is discontinuous. So, a first concern has been to understand the behavior on the non-wandering set, which may be non-invariant, and to find out sufficient conditions for the system to preserve a probability measure on the $\sigma$-algebra of the Borel sets. This issue has been successfully addressed in \cite{AC14}. The second main query in this context is the existence of probability measures suitable for specific investigations or applications. A natural way of selecting invariant measures is provided by the variational principle \cite{W82, BR75}, a relation between the topological and the measure-theoretic entropy which aims to find relevant extremal elements of the convex set of invariant probability measures.

However, the classical notion of topological entropy requires continuity of the dynamical system, a request we can no longer fulfill. Accordingly, in Section~\ref{new.topology} we introduce a generalized concept of entropy, which coincides with the classical one for continuous semiflows and is invariant under flow conjugacy. This new entropy concept turns out to be adequate to the kind of discontinuities under consideration and the right notion to establish a variational principle for impulsive semiflows.

\subsection{Topological entropy}\label{new.topology}
Here we recall the usual notion of topological entropy when $X$ is a compact metric space and $\varphi:\R^+_0 \times X\to X$ is a continuous semiflow and introduce a modified definition adapted to our setting.

\subsubsection{The classical definition}
Given  $x\in X$,  $T> 0$ and $\epsilon>0$ we define the \emph{dynamic ball}
$$B(x,\varphi, T,\epsilon)=\{ y\in X : \dist(\varphi_t (x),\varphi_t (y))<\epsilon, \text{ for every } t\in [0,T]\}.$$
The continuity of $\varphi$ implies that $B(x,\varphi, T,\epsilon)$ is an open set of $X$ since it if the open ball centered at $x$ of radius $\epsilon$ for the metric
$$\dist_T^\varphi(x,y)=\max_{0\leq t\leq T}\,\{\dist(\varphi_t(x), \varphi_t(y))\}.$$
A set $E\subseteq X$ is said to be $(\varphi, T, \epsilon)$-\emph{separated} if,
for each $x\in E$, inside the ball $B(x,\varphi, T,\epsilon)$ there is no other point of $E$ besides $x$.
As a consequence of the compactness of $X$ and the continuity of $\varphi$, any set $E\subseteq X$ which  is $(\varphi, T, \epsilon)$-separated is finite. If we denote by $|E|$ the cardinality of $E$, then we define the largest number of distinct, up to $\epsilon$, initial $T$-blocks of orbits of $\varphi$ by
$$s(\varphi, T,\epsilon)=\max\{|E| : E \text{ is } (\varphi, T, \epsilon)\text{-separated}\},$$
and the growth rate of this number as
$$h(\varphi,\epsilon)=\limsup_{T\to +\infty}\frac1T \log s(\varphi, T,\epsilon).$$
The \emph{topological entropy} of $\varphi$ is then given by
$$h_{\rm{top}}(\varphi)=\lim_{\epsilon\to 0^+}h(\varphi,\epsilon).$$

\subsubsection{A modified definition}
We now change the previous definition of topological entropy. Let $X$ be a metric space and $\psi:\R^+_0 \times X\to X$ a (not necessarily continuous) semiflow.

\begin{definition}\label{def.admissible} Consider a function $\tau$ assigning to each $x\in X$ a strictly increasing (possibly finite) sequence of positive real numbers
 $(\tau_n(x))_{n\in A(x)}$, where either $A(x)=\{1,\dots,\ell\}$ for some $\ell\in \N$ or $A(x)=\N$.
We say that $\tau$ is \emph{admissible with respect to $Z\subset X$}  if there exists $\eta>0$ such that $\tau_1(x)\geq \eta$ for all $x\in Z$, and for all $x\in X$:
\begin{enumerate}
\item $\tau_n(\psi_s(x))=\tau_n(x)-s$, for all $n\in \N$ and all $s\ge0$;
\item $\tau_{n+1}(x)-\tau_n(x)\geq\eta,$ for all $n\in \N$ with $n+1\in A(x)$.
\end{enumerate}

\end{definition}

For each admissible function $\tau$, $x\in X$, $T>0$ and $0<\delta<\eta/2$, we define
$$ J_{T,\delta}^\tau(x)= (0,T]\setminus\left(\bigcup_{j=1}^{n_T(x)}\,]\tau_j(x)-\delta,\tau_j(x)+\delta\,[\right),$$
where $n_T(x)=\max\{n\geq 1:\tau_n(x)\leq T\}$.
The \emph{$\tau$-dynamical ball} of radius $\epsilon>0$  centered at $x$ is the set
$$B^\tau(x,\psi, T,\epsilon,\delta)=\left\{y\in X:\dist(\psi_t (x),\psi_t (y))< \epsilon,\: \forall t\in J_{T,\delta}^\tau(x) \right\}.$$
Accordingly, a set $E\subseteq X$ is said to be $(\psi,\tau, T, \epsilon,\delta)$-\emph{separated} if, for each $x\in E$, we have
$$y\notin B^\tau(x,\psi, T,\epsilon,\delta), \quad\forall y\in E\setminus\{x\}.$$
As before, define 
$$
s^\tau(\psi, T,\epsilon,\delta)=\sup\{|E|:E \text{ is a finite } (\psi,\tau, T, \epsilon,\delta)\text{-separated set}\},$$
and the growth rate
$$h^\tau(\psi,\epsilon,\delta)= \limsup_{T\to +\infty}\frac1T \log s^\tau(\psi, T,\epsilon,\delta),$$
where $\log \infty =\infty$. As the function $\epsilon \mapsto h^\tau(\psi, \epsilon,\delta)$ is decreasing, the following limit exists
$$h^\tau(\psi,\delta)=
\lim_{\epsilon\to 0^+} h^\tau(\psi,\epsilon,\delta).$$
Finally, as the function $\delta \mapsto h^\tau(\psi,\delta)$ is also decreasing, we define the \emph{$\tau$-topological entropy} of $\psi$
$$h_{\rm{top}}^\tau(\psi)= \lim_{\delta\to 0^+}h^\tau(\psi,\delta).$$

\begin{maintheorem} \label{teo.coincide} Let $\varphi:\R^+_0\times X\to X$ be a continuous semiflow on a compact metric space~$X$ and $\tau$ an admissible function on $X$.  Then $h_{\rm{top}}^\tau(\varphi)=h_{\rm{top}}(\varphi).$
\end{maintheorem}

\subsection{Impulsive semiflows.}

Consider a compact metric space $X$, a continuous semiflow $\varphi:\R^+_0 \times X\to X$, a nonempty compact  set $D\subset X$ and a continuous map $I:D \to X$. Given $\xi>0$, let
\begin{equation}\label{eq.dxi}
D_\xi=\bigcup_{x\in D }\{\varphi_t(x):  0<t<\xi\}.
\end{equation}

\begin{definition}\label{def.tubecondition}
We say that $D$ satisfies a \emph{half-tube condition} if there is $\xi_0>0$ such that:
\begin{enumerate}
\item $D_\xi$ is an open set for each $0<\xi\le\xi_0$;
\item If $\varphi_t(x)\in D_{\xi_0}$ for some  $x\in X$ and $t>0$, then there is $0\le t'<t$ with $\varphi_{t'}(x)\in D$;
\item  $\{\varphi_t(x_1):  0<t<\xi_0\}\cap \{\varphi_t(x_2):  0<t<\xi_0\}=\emptyset$ for all $x_1,x_2\in D$ with $x_1\neq x_2$.
\end{enumerate}
\end{definition}

The first visit of each $\varphi$-trajectory to $D$ will be registered by the function $\tau_1:X\to~[0,+\infty]$ defined by
$$
\tau_1(x)=
\begin{dcases}
\inf\left\{t> 0:\varphi_t(x)\in D\right\} ,& \text{if } \varphi_t(x)\in D\text{ for some }t>0;\\
+\infty, & \text{otherwise.}
\end{dcases}
$$
It is known that the function $\tau_1$ is  lower semicontinuous on the set $X\setminus D$; see \cite[Theorem~2.7]{C04a}. Additionally, the \emph{tube condition} proposed in \cite{C04a} ensures that the restriction of $\tau_1$ to $X\setminus D$ is also upper semicontinuous. 

Assuming $\tau_1(x)>0$ for all $x\in X$, we define the \emph{impulsive trajectory} $\gamma_x:[0,T(x)[\,\to X$  and the subsequent \emph{impulsive times} of $x\in X$ according to the following rules:
\begin{enumerate}
\item If $0\le t<\tau_{1}(x)$, then we set  $\gamma_x(t)=\varphi_t(x)$.
\item If $\tau_1(x)<\infty$, then we proceed inductively:
\begin{enumerate}
\item Firstly we set
  $$\gamma_x(\tau_1(x))=I(\varphi_{\tau_1(x)}(x)).$$
  Defining the second impulsive time of $x$ as
  $$\tau_2(x)=\tau_1(x)+\tau_1(\gamma_x(\tau_1(x))),$$
 we set
         $$\gamma_x(t)=\varphi_{t-\tau_1(x)}(\gamma_x(\tau_1(x))),\quad\text{for }\tau_1(x)<t<\tau_2(x).$$
         \item Assuming that $\gamma_x(t)$ is defined for $t<\tau_{n}(x)$, for some $n\ge 2$, we set
  $$\gamma_x(\tau_{n}(x))=I(\varphi_{\tau_n(x)-\tau_{n-1}(x)}(\gamma_x({\tau_{n-1}(x)}))).$$
  Defining the $(n+1)^{\text{th}}$ impulsive time of $x$ as
  $$\tau_{n+1}(x)=\tau_n(x)+\tau_1(\gamma_x(\tau_n(x))),$$
  we set
         $$\gamma_x(t)=\varphi_{t-\tau_n(x)}(\gamma_x(\tau_n(x))),\quad\text{for }\tau_n(x)<t<\tau_{n+1}(x).$$
\end{enumerate}
Finally, we define the time length of the trajectory of $x$ as
 $$T(x)=\sup_{n\ge 1}\,\{\tau_n(x)\}.
 $$
\end{enumerate}
We say that $(X,\varphi,D,I)$ is an \emph{impulsive dynamical system} if
 $$\tau_1(x)>0\quad\text{and}\quad T(x)=+\infty, \quad\text{for all $x\in X$}.$$
As observed in \cite[Remark 1.1]{AC14}, under the condition $I(D)\cap (D)=\emptyset$ we have  $T(x)=\infty$ for all $x\in X$ and $\tau=\{\tau_n(x)\}_{n\geq 1}$ is an admissible function with respect to $D$.
The  \emph{impulsive semiflow} $\psi$ of an impulsive dynamical system $(X,\varphi, D, I)$ is defined by
$$
\begin{array}{cccc}
        \psi:  &  \mathbb{R}^+_0 \times X & \longrightarrow &X \\
        & (t,x) & \longmapsto & \gamma_x(t), \\
        \end{array}$$
where $\gamma_x(t)$ is the impulsive trajectory of $x$ determined by $(X,\varphi,D, I)$.
It has been  proved in  \cite[Proposition 2.1]{B07} that $\psi$ is indeed a semiflow, though not necessarily continuous.

For small enough $\xi>0$ we define
$$X_\xi=X\setminus (D_\xi\cup D).$$
Observe that, as we are assuming that $D$ satisfies a half-tube condition (see item (2) of Definition~\ref{def.tubecondition}), then  $X_\xi$ is forward invariant under $\psi$ (that is, $\psi_t(X_\xi) \subseteq X_\xi$ for all $t \geq 0$).

To control the moments a $\varphi$-trajectory visits $D$, we introduce the function
$$\tau^*:X_\xi \cup D \to [0,+\infty]$$
defined by
$$
\tau^*(x)=
\begin{cases}
\tau_1(x), &\text{if $x\in X_\xi$};\\
0, &\text{if $x\in D$}
\end{cases}
$$
and, in what follows, we will assume that $\tau^*$ is a continuous map.

\begin{definition}\label{def.transverse}
We say that $I(D)$ is \emph{transverse} if there are $s_0>0$ and $\xi_0>0$ such that
\begin{enumerate}
\item $ \varphi_t(x) \in I(D) \,\Rightarrow \,\varphi_{t+s}(x) \notin I(D),\,\,\forall\,0 < s < s_0$;
\item  $\{\varphi_t(x_1):  0<t<\xi_0\}\cap \{\varphi_t(x_2):  0<t<\xi_0\}=\emptyset$ for all $x_1,x_2\in I(D)$ with $x_1\neq x_2$.
\end{enumerate}
\end{definition}
\noindent This property holds, for instance, when $\varphi$ is a $C^1$ semiflow and $I(D)$ is transversal to the flow direction.

The map $I$ is said to be \emph{1-Lipschitz} if for all $x, y \in D$ we have
$$\dist\,(I(x),I(y))\leq \,\dist\, (x,y).$$

\begin{maintheorem}\label{teo.modelocontinuo}
Let $\psi$ be the semiflow of  an impulsive dynamical system $(X,\varphi, D, I)$ such that $I$ is 1-Lipschitz, $I(D)\cap D=\emptyset$, $D$ satisfies a half-tube condition, $I(D)$ is transverse and $\tau^*$ is continuous. Then there exist a compact metric space $\widetilde X$, a continuous semiflow $\tilde\psi$ in $\widetilde X$ and a continuous invertible bimeasurable map $h:X_\xi \to \widetilde X$ such that
$\tilde\psi_t\circ h = h\circ \psi_t$ for all $t\ge 0$ and
$$h^\tau_{\rm{top}}(\psi)=h_{\rm{top}}(\tilde\psi).$$
\end{maintheorem}

We are left to relate the topological entropy of $\tilde\psi$ with the metric entropies of the time-one map $\psi_1$ induced by the impulsive semiflow $\psi$. In the sequel, $\mathcal{M}_{\psi}(X)$ will stand for the set of probability measures defined on the $\sigma$-algebra of the Borel subsets of $X$ and invariant by the impulsive semiflow associated to the impulsive dynamical system $(X,\varphi,D,I)$.

\begin{maintheorem}\label{teo.var.principle} Let $\psi$ be the semiflow of  an impulsive dynamical system $(X,\varphi, D, I)$  satisfying the assumptions of Theorem~\ref{teo.modelocontinuo} and such that $\mathcal{M}_{\psi}(X)\neq \emptyset$.
Then
$$h^\tau_{\rm{top}}(\psi)=\sup\,\{h_\mu(\psi_1):\,\mu \in \mathcal{M}_{\psi}(X)\}.$$
\end{maintheorem}

Regarding the additional demand in the statement of the previous theorem, we recall that \cite[Theorem A]{AC14} shows that conditions $I(D)\cap D=\emptyset$ and $I(\Omega_\psi\cap D) \subset \Omega_\psi \setminus D$ together are sufficient for $\mathcal M_\psi(X)$ to be nonempty, where $\Omega_\psi$ denotes the non-wandering set of $\psi$.

In the last section, we will present a simple example satisfying the assumptions of our theorems. These results also apply, for instance, to the discontinuous local semiflows for Kurzweil equations studied in~\cite{ABFS11}.


\section{Topological entropy: classical and new}

In this section we will verify that the modified definition of topological entropy coincides with the classical one for continuous semiflows defined on compact metric spaces. We start proving that the trajectory of any point is uniformly continuous.

\begin{lemma}\label{le.alfabeta} Let $\varphi:\R^+_0\times X\to X$ be a continuous semiflow on a compact metric space $X$.
For each $\alpha>0$ there exists $\beta>0$ such that, for all $x\in X$ and all $t,u\geq 0$ with $|t-u|<\beta$, we have
$\dist(\varphi_t(x),\varphi_u(x))<\alpha.$
\end{lemma}
\begin{proof}

First notice that, as $X$ is a compact metric space, then $\varphi:[0,1]\times X\to X$ is uniformly continuous. In particular, given $\alpha>0$, there exist $\beta_0> 0$ such that for all $y\in X$  and all $t_0,u_0\in[0,1]$ with $|t_0-u_0|<\beta_0$ we have
$$\dist(\varphi_{t_0}(y),\varphi_{u_0}(y))<\alpha.$$
Take $\beta=\min\{\beta_0,1/2\}>0$. For every $t,u>0$ with $|t-u|<\beta$,  there exist an integer $n_0\geq 0$ and  $t_0,u_0\in[0,1]$ such that $t=n_0+t_0$, $u=n_0+u_0$ and $|t_0-u_0|=|t-u|<\beta\le \beta_0$. Taking $y=\varphi_{n_0}(x)$,  then
\begin{eqnarray*}
\dist(\varphi_{t}(x),\varphi_u(x))&=&\dist(\varphi_{n_0+t_0}(x),\varphi_{n_0+u_0}(x))\\
&=&\dist(\varphi_{t_0}(\varphi_{n_0}(x)),\varphi_{u_0}(\varphi_{n_0}(x)))\\
&=&\dist(\varphi_{t_0}(y),\varphi_{u_0}(y))\\
&<&\alpha.
\end{eqnarray*}
\end{proof}

\subsection{Proof of Theorem~\ref{teo.coincide}} Fix $0<\delta<\eta/2$, $\epsilon> 0$ and  $T>0$. Notice that for every $x\in X$ we have
$$B(x,\varphi, T,\epsilon)\subseteq B^\tau(x,\varphi, T,\epsilon,\delta),$$
and
$$s^\tau(\varphi, T,\epsilon,\delta)\leq s(\varphi, T,\epsilon), $$
so
$$h_{\rm{top}}^\tau(\varphi)\leq h_{\rm{top}}(\varphi).$$

Let us now prove the other inequality. Fix $T\geq 0$ and $\alpha>0$. By Lemma~\ref{le.alfabeta} there exists $\beta>0$
such that, for all $z\in X$ and all $t,u>0$ with $|t-u|<\beta$, we have
\begin{equation}\label{eq.phiz}
 \dist(\varphi_t(z),\varphi_u(z))<\alpha/4.
\end{equation}
Hence,  if $x,y\in X$ and $t\in[0,T]$ satisfy
\begin{equation}\label{eq.phit}
\dist(\varphi_t(x),\varphi_t(y))>\alpha,
\end{equation}
then, for every $u\in (t-\beta,t+\beta)$, we get
$$\dist(\varphi_t(x),\varphi_t(y))\le \dist(\varphi_t(x),\varphi_u(x))+\dist(\varphi_u(x),\varphi_u(y))+\dist(\varphi_u(y),\varphi_t(y))$$
which, together with \eqref{eq.phiz}  and \eqref{eq.phit}, implies
\begin{equation}\label{eq:contunif}
\dist(\varphi_u(x),\varphi_u(y)) >\alpha/2.
\end{equation}

Consider  now $E\subseteq X$ being $(\varphi, T, \alpha)$-separated. As $\varphi$ is continuous and each dynamical ball is contained in the corresponding $\tau$-dynamical ball, the set $E$ is finite. By definition, for every $x,y\in E$, $x\ne y$, there exists $t\in [0,T]$ such that
$$\dist(\varphi_t (x),\varphi_t (y))\geq\alpha.$$
Choose $0<\delta<\min\{\eta, \beta/2\}$ and $0<\epsilon<\alpha/2$. By \eqref{eq:contunif}, if  $u\in (t-2\delta,t+2\delta)$, then
$$\dist(\varphi_u(x),\varphi_u(y))>\alpha/2 >\epsilon.$$
 If $t\in\,J_{T,\delta}^\tau(x)$, then $y\notin  B^\tau(x,\varphi, T,\epsilon,\delta)$. Otherwise,  $J_{T,\delta}^\tau(x) \cap(t-2\delta,t+2\delta)\ne\emptyset$, and then  $y\notin  B^\tau(x,\varphi, T,\epsilon,\delta)$. So, $E$ is $(\varphi,\tau, T, \epsilon,\delta)$-separated. Consequently, for every $0<\delta<\min\{\eta, \beta/2\}$, $0<\epsilon<\alpha/2$ and $T> 0$,
$$s(\varphi, T,\alpha)\leq s^\tau(\varphi, T,\epsilon,\delta), $$
and so
$$\frac1{T}\log s(\varphi, T,\alpha)\leq \frac1{T}\log  s^\tau(\varphi, T,\epsilon,\delta).$$
Taking the upper limit as $T\to +\infty$, we get
$$h(\varphi,\alpha)\leq h^\tau(\varphi,\epsilon,\delta).$$
Now, if $\epsilon\to 0^+$,
$$h(\varphi,\alpha)\leq h^\tau(\varphi,\delta).$$
Noticing that $\beta=\beta(\alpha)$ and $\delta=\delta(\alpha)$, we deduce that, when $\alpha\to 0^+$, we have $\delta\to 0^+$, and therefore
$$h_{\rm{top}}(\varphi) \leq h_{\rm{top}}^\tau(\varphi).$$
\endproof

\subsection{Monotonicity of $h^{\tau}_{\rm{\top}}$}
Given $\tau$ and $\tau'$ two admissible functions in $X$, we say that $\tau'$ refines $\tau$, and write $\tau'\succ\tau$, if for all $x\in X$ and $n\in\N$ there exists $m\in\N$ such that $\tau_n(x)=\tau'_m(x)$. Next lemma proves that the new concept of topological entropy is monotone with respect to the refinement of admissible functions.

\begin{lemma}\label{le.refine} For any semiflow $\psi:\R^+_0\times X \rightarrow X$, if $\tau'\succ\tau$ then $h_{\rm{top}}^\tau(\psi)\geq h^{\tau'}_{\top}(\psi)$.
\end{lemma}

\begin{proof} Given $\epsilon >0$, $T>0$, $0<\delta <\eta/2$ and a finite $(\psi,\tau', T, \epsilon, \delta)$-separated subset $E$, as $\tau'\succ \tau$, the set $E$ is a $(\psi,\tau, T, \epsilon, \delta)$-separated as well. Therefore
$$s^{\tau'}(\psi, T,\epsilon,\delta)\leq s^\tau(\psi, T,\epsilon,\delta).$$
\end{proof}

Given two semiflows $\psi:\R^+_0\times X \rightarrow X$ and $\tilde{\psi}:\R^+_0\times \tilde{X} \rightarrow \tilde{X}$, acting on metric spaces $(X,d)$ and $(\tilde{X},\tilde{d})$, and two admissible functions $\tau$ and $\tilde{\tau}$ defined on $X$ and $\tilde{X}$, respectively, we say that a uniformly continuous surjective map $h:X \rightarrow \tilde{X}$ is a \emph{$(\tau,\tilde{\tau})$-semiconjugacy} between $\psi$ and $\tilde{\psi}$ if
\begin{enumerate}
\item $\tilde{\psi}_t\circ h=h\circ \psi_t$, for all $ t\geq 0$;
\item $\tilde{\tau}(h(x))=\tau(x)$, for all $x \in X.$
\end{enumerate}

\begin{lemma}\label{le.invariant} Let $h:X\to \tilde{X}$ be a finite-to-one $(\tau,\tilde{\tau})$-semiconjugacy between the semiflows $\psi$ and $\tilde{\psi}$ on $X$ and $\tilde{X}$, with admissible functions $\tau$ and $\tilde{\tau}$, respectively. Then $h^{\tau}_{\top}(\psi)\ge h^{\tilde{\tau}}_{\top}(\tilde{\psi})$.
\end{lemma}

\begin{proof}
Let $\psi:\R^+_0\times X \rightarrow X$ and $\tilde{\psi}:\R^+_0\times \tilde{X} \rightarrow \tilde{X}$ be two semiconjugate semiflows and $h$ be such a semiconjugacy. As $h$ is uniformly continuous, given $\epsilon >0$ there exists $\gamma>0$ such that
$$d(a,b)<\gamma \quad \Rightarrow\quad \tilde{d}(h(a), h(b))<\epsilon \quad \forall a,b \in X.$$
Fix $T>0$ and $0<\delta<\eta/2$, and consider a finite $(\tilde{\psi},\tilde{\tau}, T, \epsilon,\delta)$-separated set $B\subseteq \tilde{X}$. Then $A=h^{-1}(B)$ is finite, although it may have a cardinal bigger or equal than the one of $B$. Moreover, $A$ is a $(\psi,\tau, T, \gamma,\delta)$-separated set of $X$. Indeed, for all $a,b \in A$, there are $t_n \in \,J_{T,\delta}^{\tilde{\tau}}(h(a))$ and $s_n \in \,J_{T,\delta}^{\tilde{\tau}}(h(b))$ such that
$$\tilde{d}(\tilde{\psi}_{t_n}(h(a)), \tilde{\psi}_{t_n}(h(b)))\geq \epsilon \quad \text{ and } \quad \tilde{d}(\tilde{\psi}_{s_n}(h(a)), \tilde{\psi}_{s_n}(h(b)))\geq \epsilon$$
that is,
$$\tilde{d}(h\circ \psi_{t_n}(a), h\circ \psi_{t_n}(b))\geq \epsilon \quad \text{ and } \quad \tilde{d}(h\circ \psi_{s_n}(a), h\circ \psi_{s_n}(b))\geq \epsilon.$$
Therefore,
$$d (\psi_{t_n}(a), \psi_{t_n}(b))\geq \gamma \quad \text{ and } \quad d(\psi_{s_n}(a), \psi_{s_n}(b))\geq \gamma.$$
Taking into account that, by definition of semiconjugacy, $t_n \in \,J_{T,\delta}^\tau(a)$ and $s_n \in \,J_{T,\delta}^\tau(b)$, we deduce that
$$s^\tau(\psi, T,\gamma,\delta)\geq s^{\tilde{\tau}}(\tilde{\psi}, T,\epsilon,\delta).$$
When $\epsilon \rightarrow 0$, we have $\gamma=\gamma(\epsilon)\rightarrow 0$, and so we finally conclude that
$$h_{\rm{top}}^\tau(\psi) \geq h_{\rm{top}}^{\tilde{\tau}}(\tilde{\psi}).$$
\end{proof}

\section{Time and space restrictions}\label{se.timespace}

Consider a compact metric space $X$, a continuous semiflow $\varphi:\R^+_0 \times X\to X$, a compact set $D\subset X$ and a continuous map $I:D \to X$ under the assumptions of Theorem~\ref{teo.modelocontinuo}. Let  $\tau$ be the admissible function with respect to $D$ of the impulsive times associated to the impulsive semiflow $(X,\varphi, D, I)$.

It follows from the assumption that $I(D)$ is transverse (see Definition~\ref{def.transverse}) that the function that assigns to each $x\in X$ the sequence of visit times to $I(D)$, say $\theta(x):=(\theta_n(x))_{n \in \mathbb{N}}$, is an admissible function with respect to $I(D)$. Moreover, as $I(D)\cap D=\emptyset$, we may re-index the sequences $\tau(x)$ and $\theta(x)$ in order to assemble them in a unique admissible function $\tau'$, with respect to both $D$ and $I(D)$, where $\tau'_n(x)$ is either $\tau_m(x)$ or $\theta_m(x)$, for some $m$. This way, we have $\tau'\succ \tau$.

\begin{lemma}\label{le.equal1} $h_{\rm{top}}^{\tau'}(\psi)=h^\tau_{\rm{top}}(\psi).$
\end{lemma}

\begin{proof}
As $\tau'\succ \tau$, by Lemma~\ref{le.refine}, we have $h_{\rm{top}}^{\tau'}(\psi)\leq h^\tau_{\rm{top}}(\psi).$

Concerning the other inequality, we first observe that as the impulsive semiflow $\psi$ is continuous on points out of $D$, the set $I(D)$ is compact and disjoint from $D$ and $\tau^*$ is continuous, and strictly positive in the complement of $D$, we may find a compact neighborhood of $I(D)$, say
$$V=\{x \in X: \dist(x, I(D))\leq \Delta\}$$
for some small enough $\Delta>0$, such that
\begin{itemize}
\item[(a)] there exists $\rho>0$ such that
\begin{equation}\label{eq.rho}
\tau_1(x)\geq \rho \quad \forall\,\, x \in V;
\end{equation}
\item[(b)] for all $x \in V$, there is $\zeta_x\geq \eta/2$ such that the map $t \in [0,\zeta_x] \to \psi_t(x)$ is continuous;
\item[(c)] given $\alpha>0$, there exists $0< \beta < \rho$ such that, if $x\in V$ and $0\leq u\leq \beta$, then $\dist(\psi_u(x), x)<\alpha$.
\end{itemize}
Take $T>0$, $0<\epsilon< \Delta/2$, $\alpha=\epsilon/8$ and its corresponding $\beta$ given by (c) above. Let $0<\delta <\min\{\eta/2, \beta/2\}$. We already know that
$$s^{\tau'}(\psi, T,\epsilon,\delta)\leq s^\tau(\psi, T,\epsilon,\delta).$$
If $s^{\tau'}(\psi, T,\epsilon,\delta) < s^\tau(\psi, T,\epsilon,\delta)$, then $s^{\tau'}(\psi, T,\epsilon,\delta)$ is finite and we may consider a maximal $(\psi,\tau', T,\epsilon,\delta)$-separated set $E$. The set $E$ is also $(\psi,\tau, T,\epsilon,\delta)$-separated, though not maximal. Therefore, we may find $z \in X$ such that $E\cup \{z\}$ is still $(\psi, \tau,T, \epsilon, \delta)$-separated, but no longer $(\psi, \tau',T, \epsilon, \delta)$-separated. This means, in particular, that
\begin{enumerate}
\item for every $x \in E$,
$$\exists\, v=v_x \in J_{T,\delta}^\tau(x): \quad \dist(\psi_v(z),\psi_v(x))\geq\epsilon$$
and
$$\exists\, u=u_z \in J_{T,\delta}^\tau(z): \quad \dist(\psi_u(z),\psi_u(x))\geq\epsilon;$$
\item there is $e \in E$ such that either $z \in B^{\tau'}(e,\psi,T,\epsilon,\delta)$, that is,
$$\forall\, t \in J_{T,\delta}^{\tau'}(e) \quad  \dist(\psi_t(z),\psi_t(e))<\epsilon$$
in which case
$$ v_e \in J_{T,\delta}^\tau(e) \setminus J_{T,\delta}^{\tau'}(e),$$
which means that there is $\theta_j(e)$ satisfying
$$v_e \in (\theta_j(e)-\delta, \theta_j(e)+\delta);$$
or $e \in B^{\tau'}(z,\psi,T,\epsilon,\delta)$, that is,
$$\forall\, r \in J_{T,\delta}^{\tau'}(z) \quad  \dist(\psi_r(z),\psi_r(e))<\epsilon$$
and so we must have
$$u_z \in J_{T,\delta}^\tau(z) \setminus J_{T,\delta}^{\tau'}(z)$$
or, equivalently, there is $\theta_\ell(z)$ such that
$$u_z \in (\theta_\ell(z)-\delta, \theta_\ell(z)+\delta).$$
\end{enumerate}

Assume that $z \in B^{\tau'}(e,\psi,T,\epsilon,\delta)$ and consider
$$m = \theta_j(e)-\delta.$$
As $m \in J_{T,\delta}^{\tau'}(e)$, we know that
$$\dist(\psi_m(z),\psi_m(e))<\epsilon.$$
Moreover, as $\theta_j(e)-m < \delta<\beta$ and $\psi_{\theta_j(e)}(e) \in I(D)\subset V$, we have
$$\dist(\psi_m(e),\psi_{\theta_j(e)}(e))<\alpha = \epsilon/8$$
and therefore, as $0<\max\{v_e-m, |\theta_j(e)-v_e|\}<\beta$, we get
\begin{eqnarray*}
\dist(\psi_m(e),\psi_{v_e}(e)) &\leq& \dist(\psi_m(e), \psi_{\theta_j(e)}(e)) + \dist(\psi_{\theta_j(e)}(e), \psi_{v_e}(e)) \\
& < &\alpha + \alpha < \epsilon/4
\end{eqnarray*}
and
\begin{eqnarray*}
\dist(\psi_m(z),I(D)) &\leq& \dist(\psi_m(z), \psi_{\theta_j(e)}(e)) \\
& \leq& \dist(\psi_m(z), \psi_m(e)) + \dist(\psi_m(e), \psi_{\theta_j(e)}(e)) \\
& < &\epsilon + \epsilon/4 < \Delta.
\end{eqnarray*}
That is, $\psi_m(z) \in V$ and so, as $0<v_e-m<\beta$,
$$\dist(\psi_m(z), \psi_{v_e}(z)<\epsilon/4.$$
Analogously, as $v_e-m < \delta<\beta$ and $\psi_{v_e}(e) \in V$, we have
$$\dist(\psi_m(e),\psi_{v_e}(e))<\epsilon/4.$$
Thus
\begin{eqnarray*}
\epsilon &\leq& \dist(\psi_{v_e}(z),\psi_{v_e}(e)) \\
&\leq& \dist(\psi_{v_e}(z),\psi_m(z))+\dist(\psi_m(z),\psi_m(e))+\dist(\psi_m(e),\psi_{v_e}(e)) \\
&\leq& \epsilon/4 +\dist(\psi_m(z),\psi_m(e))+\epsilon/4
\end{eqnarray*}
hence
$$\dist(\psi_m(z),\psi_m(e)) \geq \epsilon/2.$$
This means that $z \notin B^{\tau'}(e,\psi,T,\epsilon/2,\delta)$.

In a similar way, we conclude that, if $e \in B^{\tau'}(z,\psi,T,\epsilon,\delta)$, then $e \notin B^{\tau'}(z,\psi,T,\epsilon/2,\delta)$. In any case, we deduce that the set $E\cup\{z\}$, which is $s^{\tau}(\psi, T,\epsilon,\delta)$-separated, is $s^{\tau'}(\psi, T,\epsilon/2,\delta)$-separated as well.

Consequently, for every $T>0$, $0<\delta<\min\{\eta/2, \beta/2\}$ and $0< \epsilon <\Delta$, we get
$$s^{\tau}(\psi, T,\epsilon,\delta) \leq s^{\tau'}(\psi, T,\epsilon/2,\delta)$$
and therefore
$$h_{\rm{top}}^{\tau}(\psi)\leq h^{\tau'}_{\rm{top}}(\psi).$$
\end{proof}

Let $a>0$ be the distance between the compact sets $D$ and $I(D)$, and assume that
\begin{equation}\label{eq.xi}
0<\xi<\min\,\{\eta/4, \xi_0/2, a/2\},
\end{equation}
where $\eta>0$ and $\xi_0>0$ are given in Definition~\ref{def.admissible} and Definition~\ref{def.tubecondition}, respectively.
The next result shows that the $\tau$ and $\tau'$-topological entropies of the semiflows $\psi$ and $\psi_{|_{X_\xi}}$ coincide.

\begin{lemma}\label{le.equal2} 
$h_{\rm{top}}^{\tau'}(\psi)=h^{\tau'}_{\rm{top}}(\psi_{|_{X_\xi}})$.
\end{lemma}

\proof

As $X_\xi \subset X$, 
then
$$h^{\tau'}_{\rm{top}}(\psi_{|_{X_\xi}})\leq h^{\tau'}_{\rm{top}}(\psi).$$
We are left to prove the other inequality.

Take $\epsilon >0$, $T>0$, $0<\delta <\eta/4$ and a finite $(\psi, {\tau'},T, \epsilon, \delta)$-separated set $E\subseteq X$. Let
$$A=E\cap (D\cup D_\xi) \quad \text{ and } \quad B=E\cap X_\xi.$$
The set $B$ is $(\psi_{|_{X_\xi}},{\tau'}, T,\epsilon,\delta)$-separated in $X_\xi$, and so its cardinal is smaller than $s^{\tau'}(\psi_{|_{X_\xi}}, T,\epsilon,\delta)$. We claim that the cardinal of $A$ is also bounded by $s^{\tau'}(\psi_{|_{X_\xi}}, T,\epsilon,\delta)$. Indeed, for each pair of points $a,b \in A$, take $u \in J_{T,\delta}^{\tau'}(a)$ and $v \in J_{T,\delta}^{\tau'}(b) $ such that
$$\dist(\psi_u(a),\psi_u(b))\geq\epsilon \quad \text{ and } \quad \dist(\psi_v(a),\psi_v(b))\geq\epsilon$$
that is,
$$\dist(\psi_{u-\xi}(\psi_{\xi}(a)),\psi_{u-\xi}(\psi_{\xi}(b)))\geq\epsilon \quad \text{ and } \quad \dist(\psi_{v-\xi}(\psi_{\xi}(a)),\psi_{v-\xi}(\psi_{\xi}(b)))\geq\epsilon.$$
Let
$$S=\min\,\{t \in J_{T,\delta}^{\tau'}(x): x \in A\}.$$
By Definition~\ref{def.admissible} and condition (\ref{eq.xi}), we have $S \geq \eta/2 - \delta > \eta/4$ and so $\xi < S$. Hence, as ${\tau'}$ is admissible (check item (1) of Definition~\ref{def.admissible}),
$$u-\xi \in J_{T,\delta}^{\tau'}(\psi_\xi(a)) \quad \text{ and } \quad v-\xi \in J_{T,\delta}^{\tau'}(\psi_\xi(b)).$$
Moreover, the points $\psi_{\xi}(a)$ and $\psi_{\xi}(b)$ are not in $D_\xi \cup D$.
Thus the set $\psi_\xi(A)$ is $(\psi_{|_{X_\xi}},{\tau'}, T,\epsilon,\delta)$-separated and so
$$|\psi_\xi(A)|\leq s^{\tau'}(\psi_{|_{X_\xi}}, T,\epsilon,\delta).$$
Finally, by item (3) of Definition~\ref{def.tubecondition}, $|\psi_\xi(A)|=|A|$. Therefore
$$|E|= |A|+|B| = |\psi_{\xi}(A)|+|B| \leq 2\,s^{\tau'}(\psi_{|_{X_\xi}}, T,\epsilon,\delta)$$
and so
$$s^{\tau'}(\psi, T,\epsilon,\delta) \leq 2\,s^{\tau'}(\psi_{|_{X_\xi}}, T,\epsilon,\delta).$$
Letting $T \to +\infty$, this inequality implies that
$$h^{\tau'}(\psi,\epsilon,\delta) \leq h^{\tau'}(\psi_{|_{X_\xi}},\epsilon,\delta)$$
and consequently, as $\epsilon, \delta$ may be chosen arbitrarily small, this last inequality yields
$$h_{\rm{top}}^{\tau'}(\psi) \leq h_{\rm{top}}^{\tau'}(\psi_{|_{X_\xi}}).$$
\endproof

\begin{remark}\label{re.setminus}
As we are assuming that $I(D)$ is transverse, a similar argument proves that $h_{\rm{top}}^{\tau'}(\psi)=h^{\tau'}_{\rm{top}}(\psi_{|_{X_\xi \setminus I(D)}})$.
\end{remark}

\section{A quotient space}\label{sec.quotient}

Given  an impulsive dynamical system $(X,\varphi,D,I)$, consider the quotient  space $ X/_\sim$ endowed with the quotient topology, where $\sim$ is the equivalence relation given by
 $$x\sim y\quad \Leftrightarrow \quad x=y, \quad y=I(x),\quad x=I(y)\quad \text{or}\quad I(x)=I(y).$$
Let $\pi:X\to X/_\sim$ be the natural projection.

\subsection{The induced metric}
If $d$ denotes the metric on $X$, the metric $\tilde d$ in $\pi(X)$ that induces the quotient topology is given by
$$ \tilde d \, (\tilde x,\tilde y) = \inf \, \{d\,(p_1, q_1)+ d\,(p_2, q_2)+\cdots +d\,(p_n,q_n)\},$$
where $p_1, q_1, \dots, p_n, q_n$ is any chain of points in $X$ such that
$ p_1 \sim x$, $q_1 \sim p_2$, $q_2 \sim p_3$, ... $q_n \sim y$; see $\S$23 of \cite{W70}.
In particular, we have
$$\tilde d \, (\tilde x,\tilde y) \leq d\,(x,y), \quad \forall \text{$x,y\in X$.}$$
Yet, the length $n$ of the chains needed to evaluate $\tilde d \, (\tilde x,\tilde y)$ may be arbitrarily large, preventing us from comparing $\tilde d \, (\tilde x,\tilde y)$ with $d\,(p,q)$ for all $p \sim x$ and $q \sim y$. This difficulty is overcome if we are able to uniformly bound the range of $n$; this is feasible, for instance, when the map $I$ does not expand distances.

\begin{lemma}\label{le.chain}
If $I$ is 1-Lipschitz, then for all $\tilde x, \tilde y \in \pi(X)$ there exist $p, q \in X$ such that
$$p \sim x, \quad  q \sim y \quad \text{ and } \quad d(p,q) \leq 2\,\tilde d\,(\tilde x, \tilde y).$$
\end{lemma}

\begin{proof} 
We will show that
$$\forall \, \tilde x, \tilde y \in \pi(X) \quad \tilde d \, (\tilde x,\tilde y) = \inf \, \{d\,(p, q): p\sim x, q \sim y\}.$$
Clearly, for all $\tilde x, \tilde y \in \pi(X)$, we have
$$ \tilde d \, (\tilde x,\tilde y) \leq \inf \, \{d\,(p, q): p\sim x, \,q \sim y\}.$$
Conversely, take a chain $p_1, q_1, \dots, p_n, q_n \, \in X$ such that
$$ p_1 \sim x,\, q_1 \sim p_2, \,q_2 \sim p_3, \,.. \,q_n \sim y.$$
\begin{enumerate}
\item If $q_1=p_2$, then
$$d\,(p_1,q_1) + d\,(p_2,q_2) = d\,(p_1,q_1) + d\,(q_1,q_2)\geq d\,(p_1,q_2).$$
\item If $p_2=I(q_1)$,  then
$$d\,(p_1,q_1) + d\,(p_2,q_2) \geq d\,(I(p_1),I(q_1)) + d\,(p_2,q_2)\geq d\,(I(p_1),q_2).$$
\item If $q_1=I(p_2)$, then
$$d\,(p_1,q_1) + d\,(p_2,q_2) \geq d\,(p_1,q_1) + d\,(I(p_2),I(q_2))\geq d\,(p_1,I(q_2)).$$
\item If $I(q_1)=I(p_2)$, then
$$d\,(p_1,q_1) + d\,(p_2,q_2) \geq d\,(I(p_1),I(q_1)) + d\,(I(p_2),I(q_2)) \geq d\,(I(p_1),I(q_2)).$$
\end{enumerate}
As $I(q_2) \sim q_2$, we may proceed by induction on $n$, thus concluding that there are $P,Q  \in X$ such that $P \sim x$, $Q \sim y$ and
$$d\,(p_1,q_1) + \cdots + d\,(p_n,q_n) \geq d\,(P,Q).$$
Therefore
$$ \tilde d \, (\tilde x,\tilde y) \geq \inf \, \{d\,(p, q): p\sim x, q \sim y\}.$$
Having proved that
$$ \tilde d \, (\tilde x,\tilde y) = \inf \, \{d\,(p, q): p\sim x, q \sim y\},$$
we may find $p \sim x$ and $q \sim y$ such that $d(p,q) \leq 2\,\tilde d\,(\tilde x, \tilde y).$
\end{proof}

\subsection{An induced semiflow}
Assuming that $I(D)\cap D = \emptyset$, then each point in the set $X_\xi=X\setminus (D_\xi\cup D)$ has a representative of the same equivalence class in $X\setminus D_\xi $. This implies that
\begin{equation}\label{eq.xxi}
\pi(X_\xi)=\pi(X_\xi\cup D)
\end{equation}
and, by the half-tube condition (see item (1) of Definition~\ref{def.tubecondition}), this is a compact set. In particular, $\pi(X_\xi)$ with the quotient topology is a compact metric space: indeed, as $X_\xi\cup D$ is a compact metric space and \eqref{eq.xxi} holds, then $\pi(X_\xi)$ is a compact pseudometric space; moreover, as $D$ is compact and $I:D\to X$ is continuous, $\pi(X_\xi)$ is a $T_0$ space, and so the quotient topology in $\pi(X_\xi)$ is given by a metric; see \cite{AC14} for more details.

For any $x,y\in X_\xi$ we have $x\sim y$ if and only if $x=y$.
This shows that $\pi\vert_{X_\xi}$ is a continuous bijection (not necessarily a homeomorphism) from $X_\xi$ onto $\pi(X_\xi )$.
Then, setting
\begin{equation}\label{def.psi}
\tilde\psi(t,\tilde x)=\pi(\psi(t,x))
\end{equation}
for
each $x\in X_\xi$ and $t\ge 0$,
we have that
$$\tilde\psi:\R^+\times \pi(X_\xi )\to \pi(X_\xi)$$
 is well defined and
obviously satisfies for all $t\ge0$
\begin{equation}\label{eq.conjuga2}
\tilde\psi_t\circ \pi\vert_{X_\xi} = \pi\circ \psi_t\vert_{X_\xi}.
\end{equation}
In what follows we will show that $\tilde\psi$ is continuous.

\begin{lemma}\label{le.continuity}
Assume that  $\tau^*$ is continuous and $I(D)\cap D=\emptyset$. Then
$\pi\circ\psi_t\vert_{X_\xi}$ is continuous for all $t\ge 0$.
\end{lemma}

\begin{proof}
Given $t>0$,  let us prove the continuity of $\pi\circ\psi_t\vert_{X_\xi}$ at any point $x\in X_\xi$. By an inductive argument on the impulsive times of $x$, it is enough to show that, when $y\in X_\xi$ is close to $x$, then $\pi(\psi_s(y))$ remains close to $\pi(\psi_s(x))$ for all $0\le s\le \tau_1(x)$. Notice that such an inductive argument on the impulsive times can be  applied because we are sure that $I(D)\subset X_\xi \setminus D$. The proof follows according to several cases:\\
\smallskip

\noindent\emph{Case 1}. $\tau_1(x)>t$.
\smallskip

\noindent As $\tau^*$ is continuous and $\tau_1$ coincides with $\tau^*$ in $X_\xi$, we must have  $\tau_1(y)>t$ for any point $y\in X_\xi$ sufficiently close to $x$. Therefore, the result follows in this case from the continuity of the semiflow $\varphi$.

\smallskip
\noindent\emph{Case 2}.  $\tau_1(x)  \le  t$.
\smallskip

\noindent
 Given $y\in X_\xi$ sufficiently close to $x$, by the continuity of the semiflow $\varphi$ the $\psi$-trajectories of $x$ and $y$ remain close until one of them hits the set $D$. At this moment the impulsive function acts and, therefore, their $\psi$-trajectories may not remain close at this first impulsive time. Now we distinguish three possible subcases:

\smallskip
\noindent\emph{Subcase 2.1}.  $\tau_1(x)=\tau_1(y)$.
\smallskip

\noindent
 The continuous map $I$ keeps the points $I(\varphi_{\tau_1(x)}(x)) $ and $ I(\varphi_{\tau_1(x)}(y))$ close, and this implies that  $\psi_s(x)$ and $\psi_s(y)$  remain close for all $0\le s\le \tau_1(x)$.

\smallskip
\noindent\emph{Subcase 2.2}.  $\tau_1(x)< \tau_1(y)$.
\smallskip

\noindent  By the continuity of $\varphi$ we have $\varphi_s(y)$ close to $\varphi_s(x)$ for~$y$ sufficiently close to~$x$ and $0\le s\le\tau_1(x)$.  This in particular implies that $\psi_s(y)$ is close to $\psi_s(x)$  for $0\le s<\tau_1(x)$. It remains to check that $\pi(\psi_{\tau_1(x)}(y))$  is close to $\pi(\psi_{\tau_1(x)}(x))$. This is clearly true because
 $\varphi_{\tau_1(x)}(y)$ is close to  $\varphi_{\tau_1(x)}(x)$, and so
 $$\pi(\psi_{\tau_1(x)}(y))=\pi(\varphi_{\tau_1(x)}(y))$$ is close to
 $$\pi(\varphi_{\tau_1(x)}(x)) =\pi(I(\varphi_{\tau_1(x)}(x)))=\pi(\psi_{\tau_1(x)}(x)).$$

\smallskip
\noindent\emph{Subcase 2.3}.  $\tau_1(x)> \tau_1(y)$.
\smallskip

\noindent
Again, by the continuity of $\varphi$, we have 
$\psi_s(y)$ is close to $\psi_s(x)$  for $0\le s<\tau_1(y)$. We are left to verify that $\pi(\psi_{s}(y))$  is close to $\pi(\psi_{s}(x))$ for $\tau_1(y)\le s\le \tau_1(x)$.

By the definition of  first impulsive time we have $\varphi_{\tau_1(y)}(y) \in D$; so, as we are assuming that $I(D)\cap D= \emptyset$ and have chosen $\xi < a/2$, we know that
$\psi_{\tau_1(y)}(y)=I(\varphi_{\tau_1(y)}(y)) \in X_\xi,$ which, by \eqref{eq.rho}, yields
$$\tau_1(\psi_{\tau_1(y)}(y))\geq \rho.$$
Using that $\tau^*$ is continuous at~$x$, we have
$\tau^*(x)-\tau^*(y) $ small for $y$ close to $x$; we may ensure, in particular, that
$$\tau^*(x)-\tau^*(y) <\rho.$$
Hence, for $\tau_1(y)\le s\le\tau_1(x)$, we have
$$\psi_s(y)=\varphi_{s-\tau_1(y)}(\psi_{\tau_1(y)}(y))=\varphi_{s-\tau_1(y)}(I(\varphi_{\tau_1(y)}(y))).$$
Observing that $s-\tau_1(y)\le \tau_1(x)-\tau_1(y)$ is close to $0$ for $y$ close to $x$, we have
$$\varphi_{s-\tau_1(y)}(I(\varphi_{\tau_1(y)}(y)))\quad\text{ close to }\quad I(\varphi_{\tau_1(y)}(y)).$$
Hence for $\tau_1(y)\le s\le\tau_1(x)$ we have
$$\pi(\psi_s(y))\quad\text{ close to }\quad\pi(I(\varphi_{\tau_1(y)}(y)))=\pi(\varphi_{\tau_1(y)}(y)).$$
Now we just need to notice that,  for $\tau_1(y)\le s\le\tau_1(x)$,  we have
$\varphi_{\tau_1(y)}(y)$ close to
$\varphi_{s}(y)$ which is itself close to $\varphi_{s}(x)$.
This way, we get, for $\tau_1(y)\le s\le\tau_1(x)$,
 $$\pi(\varphi_{\tau_1(y)}(y)) \quad\text{ close to }\quad  \pi(\varphi_{s}(x))= \pi(\psi_{s}(x)).$$
Lastly, recall that for $s=\tau_1(x)$ we have $  \pi(\varphi_{\tau_1(x)}(x))=\pi(I(\varphi_{\tau_1(x)}(x)))= \pi(\psi_{\tau_1(x)}(x)). $
\end{proof}

\begin{proposition}\label{prop.continuity}
The semiflow  $\tilde\psi:\R^+_0 \times \pi(X_\xi)\to \pi(X_\xi)$ is continuous.
\end{proposition}

\begin{proof}

Considering for each $\tilde x \in \pi(X_\xi)$ the map $\tilde\psi^{\tilde x}: \R^+_0\to \pi(X_\xi)$ defined by
 $$\tilde\psi^{\tilde x}(t)=\tilde\psi(t,\tilde x),$$
it is enough to prove that $\tilde\psi^{\tilde x}$ and $\tilde\psi_t$ are continuous for all $\tilde x\in \pi(X_\xi)$ and all $t\ge0$.

Let  us start by proving the continuity of $\tilde\psi^{\tilde x}$ for   $x\in X_\xi$. Take first $t_0\ge0$  which is not an impulsive time for $x$. In this case we have, for $t$ in a sufficiently small neighborhood of $t_0$ in $\R^+_0$,
 $$\tilde\psi^{\tilde x}(t)=\pi(\varphi(t,x))$$
and, as $\varphi$ is continuous, this obviously gives the continuity of $\tilde\psi^{\tilde x}$ at $t_0$. On the other hand, if $t_0$ is an impulsive time for $x$, then we have
 $$\lim_{t\to t_0^-}\tilde\psi^{\tilde x}(t)=\lim_{t\to t_0^-}\pi(\psi(t,x))=\lim_{t\to t_0^-}\pi(\varphi(t,x))
= \pi(\varphi(t_0,x)).$$
 As $\varphi(t_0,x)\in D$,  it follows from the definition of $\psi(t_0,x)$  and the equivalence relation that yields the projection $\pi$ that
   $$ \pi(\varphi(t_0,x))=\pi(I(\varphi(t_0,x)))=\pi(\psi(t_0,x))=\tilde\psi^{\tilde x}(t_0).$$
This gives the continuity of $\tilde\psi^{\tilde x}$ on the left hand side of $t_0$. The continuity on the right hand side of $t_0$ follows easily from the fact that, by definition, the impulsive trajectories  are continuous on the right hand side.

Let us now prove the continuity of $\tilde\psi_t$ for $t\ge0$. Notice that as we are considering the quotient topology in $\pi(X_\xi)$, we know that $\tilde\psi_t$ is
continuous if and only if $\tilde\psi_t\circ\pi\vert_{X_\xi}$ is continuous. The continuity of $\tilde\psi_t\circ\pi\vert_{X_\xi}$ is an immediate consequence of Lemma~\ref{le.continuity} and \eqref{eq.conjuga2}.
\end{proof}

\subsection{Proof of Theorem~\ref{teo.modelocontinuo}}
We take $\widetilde X=\pi(X_\xi)$, the semiflow $\tilde\psi$ as given in~\eqref{def.psi} and the map $h: X_\xi\to \widetilde X$ given by $h(x)=\pi(x)$ for all $x\in X_\xi$. It follows from \eqref{eq.conjuga2} and Proposition~\ref{prop.continuity} that $\tilde\psi_t\circ h = h\circ \psi_t$ for all $t\ge 0$. Thus, we are left to prove that
$$h^\tau_{\rm{top}}(\psi)=h_{\rm{top}}(\tilde\psi).$$

Let  $f:X_\xi\to \pi(X_\xi)$ be given by $f(x)=\pi(x)$. The map $f$  is a continuous bimeasurable (see \cite{P66}) bijection, so, using it, we define admissible functions $\tilde\tau$ and $\tilde\tau'$ for the semiflow $\tilde\psi$ as
$$\tilde\tau(f(x))=\tau(x) \quad \text{and} \quad \tilde\tau'(f(x))=\tau'(x).$$
Notice that $f$ is  a $(\tau,\tilde\tau)$-semiconjugacy between $\psi_{|_{X_\xi}}$ and $\tilde\psi$.

As  $I^{-1}(D)=\emptyset$, we have, for all $x \in D$,
$$\pi^{-1}(\{\pi(x)\})=\{x, I(x)\}\cup I^{-1}(\{I(x)\}).$$
So, we may restrict $f$ to $X_\xi\setminus I(D)$ and define the map
$$g:X_\xi\setminus I(D) \to \pi(X_\xi)\setminus \pi(D), \quad g(x)=f(x).$$
Observe that, as $g$ is a restriction of $\pi$ and $\pi$ is uniformly continuous on $X$, then $g$ is uniformly continuous as well.

\begin{lemma}\label{le.proj}  If $I$ is 1-Lipschitz, then $g^{-1}$ is uniformly continuous.
\end{lemma}

\begin{proof}
As in $X_\xi \setminus I(D)$ each equivalence class in $X$ has only one member, we deduce from Lemma~\ref{le.chain} that, for all $\tilde{x},\tilde{y} \in \pi(X_\xi)\setminus \pi(D)$, we have
$$d(x,y)\leq 2\,\tilde d\,(\tilde x, \tilde y).$$
This in turn implies that $g^{-1}$ is uniformly continuous.
\end{proof}

After Lemma~\ref{le.proj}, we define
$$ \tilde\tau'(g(x))=\tau'(x)$$
and this way $g$ is a $(\tau',\tilde\tau')$-semiconjugacy between $\psi_{|_{X_\xi\setminus I(D)}}$ and $\tilde\psi_{|_{\pi(X_\xi)\setminus \pi(D)}}$,
and $g^{-1}$ is a $(\tilde\tau',\tau')$-semiconjugacy between $\tilde\psi_{|_{\pi(X_\xi)\setminus \pi(D)}}$ and $\psi_{|_{X_\xi\setminus I(D)}}$.

\begin{lemma}\label{le.equal3} $h^{\tilde\tau'}_{\rm{top}}(\tilde\psi)=h_{\rm{top}}^{\tau'}(\psi_{|_{X_\xi}}).$
\end{lemma}

\begin{proof} From  Lemma~\ref{le.invariant} applied to the semiconjugacy $f:X_\xi \to \pi(X_\xi)$ we deduce that
$$h^{\tilde\tau'}_{\rm{top}}(\tilde\psi)\leq h^{\tau'}_{\rm{top}}(\psi_{|_{X_\xi}}).$$

Conversely, Lemma~\ref{le.proj} ensures that the map $g^{-1}$ is a $(\tilde\tau',\tau')$-semiconjugacy between $\tilde\psi_{|_{\pi(X_\xi)\setminus \pi(D)}}$ and $\psi_{|_{X_\xi\setminus I(D)}}$, and so we get from Lemma~\ref{le.invariant}
$$h^{\tilde\tau'}_{\rm{top}}(\tilde\psi{|_{\pi(X_\xi)\setminus \pi(D)}})\geq h^{\tau'}_{\rm{top}}(\psi_{|_{X_\xi \setminus I(D)}}).$$
As $h^{\tau'}_{\rm{top}}(\psi_{|_{X_\xi \setminus I(D)}})=h^{\tau'}_{\rm{top}}(\psi_{|_{X_\xi}})$ (see Remark~\ref{re.setminus}), we conclude that
$$h^{\tilde\tau'}_{\rm{top}}(\tilde\psi)\geq h^{\tilde\tau'}_{\rm{top}}(\tilde\psi{|_{\pi(X_\xi)\setminus \pi(D)}})\geq h^{\tau'}_{\rm{top}}(\psi_{|_{X_\xi \setminus I(D)}})=h^{\tau'}_{\rm{top}}(\psi_{|_{X_\xi}}) .$$
\end{proof}

\begin{lemma}\label{le.entropy}
$h_{\rm{top}}(\tilde\psi)=h^{\tau'}_{\rm{top}}(\psi).$
\end{lemma}

\begin{proof}
Firstly, by Theorem~\ref{teo.coincide} and Lemma~\ref{le.equal3} we get
$$h_{\rm{top}}(\tilde\psi)= h^{\tilde\tau'}_{\rm{top}}(\tilde\psi)=h_{\rm{top}}^{\tau'}(\psi_{|_{X_\xi}}).$$
Then, by Lemma~\ref{le.equal2} we obtain
$$h^{\tau'}_{\rm{top}}(\psi_{|_{X_\xi}})= h^{\tau'}_{\rm{top}}(\psi).$$
\end{proof}

To conclude the proof of Theorem~\ref{teo.modelocontinuo}, we have just to notice that from Lemma~\ref{le.equal1} we get
$$h^{\tau}_{\rm{top}}(\psi)=h^{\tau'}_{\rm{top}}(\psi)$$
and by Lemma~\ref{le.entropy} we deduce that
$$h^{\tau'}_{\rm{top}}(\psi)=h_{\rm{top}}(\tilde\psi).$$

\subsection{Proof of Theorem~\ref{teo.var.principle}}

Firstly, Theorem~\ref{teo.modelocontinuo} ensures that
$$h^{\tau}_{\rm{top}}(\psi)=h_{\rm{top}}(\tilde\psi).$$
Additionally, by Proposition~\ref{prop.continuity}, we may apply the Variational Principle \cite{BR75,W82} to $\tilde\psi$, getting
$$h_{\rm{top}}(\tilde\psi)=\sup\,\{h_\nu(\tilde\psi_1): \nu \in \mathcal{M}_{\tilde\psi}(\pi(X_\xi))\}.$$
To conclude the proof of Theorem~\ref{teo.var.principle}, we are due to connect the measure theoretical information of $\tilde\psi$ to the corresponding one of $\psi$, and to ascertain that we may replace $X_\xi$ by $X$ in the previous equality. Accordingly, we will start verifying that the space restriction $X_\xi=X\setminus (D_\xi\cup D)$ is negligible within the measure theoretical context we are dealing with.

\begin{lemma}\label{le.naoerrante} Let $\mu$ be a probability measure invariant by the semiflow $\psi$.
\begin{itemize}
\item[(a)] If $I(D) \cap D)=\emptyset$, then $\mu(D)=0$.
\item[(b)] If $I(D)\cap D=\emptyset$ and $D$ satisfies a half-tube condition, then $\mu(D_\xi)=0$.
\end{itemize}
\end{lemma}

\begin{proof}
Supposing that $\mu(D)>0$, it follows from Poincar\'e Recurrence Theorem that for $\mu$ almost every $x\in D$ there are infinitely many moments $t>0$ such that $\psi_t(x)\in  D.$ Clearly, if  $I(D)\cap D=\emptyset$, then the $\psi$-trajectories do not hit $D$ for $t>0$, and so we arrive at a contradiction.

Assume that $\mu(D_\xi)>0$. Then, for $\mu$ almost every $x \in D_\xi$, there are infinitely many times $t>0$ such that $\psi_t(x) \in D_\xi$. Yet, as $I(D)\cap D=\emptyset$ and $D$ satisfies a half-tube condition (see item (2) of Definition~\ref{def.tubecondition}), no $\varphi$-trajectory enters $D_\xi$ unless it has previously crossed $D$; and, once at $D$, it is sent, by the impulsive semiflow $\psi$, to $I(D)$. Thus, no $\psi$-trajectory comes back to $D_\xi$ for $t>0$; this way we reach a contradiction.
\end{proof}

We remark that, as $D_\xi$ is an open set, the proof of the previous lemma also shows that $\Omega_\psi\subset X_\xi \cup D$.

Let us now exchange ergodic data between $\tilde\psi$ and $\psi$. Consider the continuous bimeasu\-rable bijection
$$
\begin{array}{cccc}
       f :&X_\xi & \longrightarrow &\pi(X_\xi) \\
         &x & \longmapsto & \pi(x)
        \end{array}$$
and the inclusion map ${i}:X_\xi \to X$. In the next two lemmas we follow the strategies used to prove \cite[Lemmas 5.2 \& 5.3]{AC14}.

\begin{lemma}\label{eq.conjpush}
$({i}\circ f^{-1})_*: \mathcal M_{\widetilde\psi}(\pi(X_\xi))  \longrightarrow \mathcal M_\psi(X_\xi)$ is well defined and is a bijection.
 \end{lemma}

\begin{proof}
To see that $ ({i}\circ f^{-1})_*$ is well defined, we need to check that if $\nu\in \mathcal M_{\widetilde\psi}(\pi(X_\xi))$, then one necessarily has $ ({i}\circ f^{-1})_*\nu\in M_\psi(X_\xi)$. Now, from
$$f\circ\psi_t={\widetilde\psi}_t\circ f,\quad\text{for all $t\ge0$},$$
 or equivalently
\begin{equation}\label{eq.menosum}
 \psi_t\circ f^{-1}=f^{-1}\circ{\widetilde\psi}_t,\quad\text{for all $t\ge0$}
\end{equation}
we clearly have that
$$\nu\in \mathcal M_{{\widetilde\psi}}(\pi(X_\xi)) \quad\Rightarrow\quad (f^{-1})_*\nu \in \mathcal M_{\psi}(X_\xi) \quad\Rightarrow\quad {i}_*(f^{-1})_*\nu\in \mathcal M_{\psi}(X). $$
Finally, as ${i}_*(f^{-1})_*\nu=({i}\circ f^{-1})_*\nu$, we conclude that $({i}\circ f^{-1})_*\nu\in \mathcal M_{\psi}(X)$. This shows that $ ({i}\circ f^{-1})_*$ is well defined.

It remains to check that $({i}\circ f^{-1})_*$ is bijective. As $({i}\circ f^{-1})_*={i}_*\circ f^{-1}_*$ and
$(f^{-1})_*$ is invertible, we only need to prove that ${i}_*$ is invertible. Clearly, being injective, ${i}$ has a left inverse; this implies that ${i}_*$ has a left inverse. Thus ${i}_*$ is  injective as well.

To prove that ${i}_*$ is surjective, given $\mu\in \mathcal M_{\psi}(X)$, let $\nu$ be the restriction of $\mu$ to the Borel subsets of $X_\xi$. Noticing that the support of $\mu$ is contained in $\Omega_\psi$, that $\Omega_\psi \subseteq X_\xi \cup D$ and that $\mu(D)=0$, we know that $\nu\in \mathcal M_{\psi}(X_\xi)$. Using the $\psi$-invariance of $\mu$, we also deduce that, for any Borel set $A\subset X_\xi$,
$$
\nu(\psi_t^{-1}(A))=\mu(\psi_t^{-1}(A)\cap(X_\xi))=\mu(\psi_t^{-1}(A))=\mu(A)=\mu(A\cap(X_\xi))=\nu(A).
$$
Moreover, for  any Borel set $A\subset X$, we have
$$
{i}_*\nu(A)=\nu({i}^{-1}(A))=\nu(A\cap(X_\xi))=\mu(A\cap(X_\xi))=\mu(A).
$$
Consequently, $\nu\in \mathcal M_{\psi}(X_\xi)$ and ${i}_*\nu=\mu$.
\end{proof}

So, Lemma~\ref{eq.conjpush} ensures that
$$h_{\rm{top}}(\tilde\psi) = \sup\,\{h_\mu(\psi_1): \mu \in \mathcal{M}_{\psi}(X_\xi)\}.$$
Besides, from Lemma~\ref{le.naoerrante} we get
$$\sup\,\{h_\mu(\psi_1): \mu \in \mathcal{M}_{\psi}(X_\xi)\} = \sup\,\{h_\mu(\psi_1): \mu \in \mathcal{M}_{\psi}(X)\}.$$
Hence,
$$h^{\tau}_{\rm{top}}(\psi)=\sup\,\{h_\mu(\psi_1):\,\mu \in \mathcal{M}_{\psi}(X)\}.$$

\bigskip

\section{An example}

Consider the phase space
$$X=\left\{(r\cos\theta,r\sin\theta)\in\R^2:  1\le r\le 2  , \,\theta \in [0,2\pi]\right\}$$
and define $\varphi:\R^+_0\times X\to X$ as the semiflow of the vector field in $X$ given by
$$
   \begin{cases}
   r'=0 & \\
   \theta'=1 .&
   \end{cases}
$$
The trajectories of $\varphi$ are circles spinning counterclockwise around zero. Take now
$$D=\{(r,0)\in X: 1\le r \le 2\}$$
and the map
$$I: (r,0) \in D \quad \mapsto \quad I(r,0)=\left(-\frac12-\frac12\,r, 0\right)$$
whose Lipschitz constant is $1/2$.

The non-wandering set of the semiflow $\psi$ of the impulsive dynamical system $(X,\varphi, D, I)$ is
$$\Omega_\psi=\left\{(\cos\theta, \sin\theta): \pi\le \theta\le 2\pi\right \}$$
and, by \cite[Theorem A]{AC14}, $\psi$ has some invariant probability measure. For a suitably small $\xi>0$,
\begin{eqnarray*}
D_\xi &=& \left\{(r\cos\theta, r\sin\theta): 1\le r\le 2, 0 < \theta< 2\pi\xi \right\} \\
X_\xi &=& \left\{(r\cos\theta,r\sin\theta)\in\R^2:  1\le r\le 2, \,\theta \in [2\pi\xi, 2\pi) \right\}.
\end{eqnarray*}

As regards the requirements of Theorem~\ref{teo.var.principle}, we have, $I(D)\cap D=\emptyset$; the set $D$ satisfies a half-tube condition (with $0<\xi_0\leq 1/2$); the function $\tau^*:X_\xi \cup D \to [0,2\pi]$ is given by
$$
\tau^*(x)=
\begin{cases}
2\pi - \theta, &\text{if $x=(r\cos\theta, r\sin\theta) \in X_\xi$};\\
0, &\text{if $x\in D$}
\end{cases}
$$
so it is continuous; whenever $\psi_t(x) \in I(D)$, we have $\{\psi_{t+s}(x): 0< s < \pi\}\cap I(D)=\emptyset;$ and, finally, $I(\Omega_\psi\cap D)=I(\{(-1,0), (1,0)\})=\{(-1,0)\} \subset \Omega_\psi \setminus D$.

The equality (\ref{eq.conjuga2}) and Proposition~\ref{prop.continuity} build a conjugacy between $\psi_{|_{\Omega_\psi \setminus D}}$ and the semiflow $\tilde\psi$ on the quotient structure where it acts as a circle rotation. Therefore
$$h^\tau_{\rm{top}}(\psi)=\sup\,\{h_\mu(\psi_1):\,\mu \in \mathcal{M}_{\psi}(X)\}=0.$$

\end{document}